\documentclass[12pt,leqno,a4paper,article]{amsart}
\usepackage{amsmath, amsthm, verbatim, amsfonts, amssymb, color}
\usepackage{mathrsfs}
\usepackage{dsfont}
\usepackage[a4paper,margin=3cm]{geometry}

\theoremstyle{plain}
\newtheorem{theorem}{Theorem}[section]

\newtheorem{lemma}[theorem]{Lemma}
\newtheorem{proposition}[theorem]{Proposition}

\theoremstyle{definition}

\newtheorem*{h1}{(H1)}
\newtheorem*{h2}{(H2)}
\newtheorem*{notation}{Notation}
\newtheorem*{ack}{Acknowledgement}

\numberwithin{equation}{section}
\renewcommand\labelenumi{\textup{\alph{enumi})}}
\renewcommand\theenumi\labelenumi

\newcommand\Ee{\mathds{E}}

\newcommand\nat{\mathds{N}}
\newcommand\real{{\mathds{R}}}
\newcommand\rd{{\mathds{R}^d}}

\newcommand\I{\mathds{1}}

\newcommand\dup{\mathrm{d}}
\newcommand\dif{\mathrm{d}}
\newcommand\eup{\mathrm{e}}
\newcommand\iup{\mathop{\mathrm{i}}}

\newcommand\Lip{\mathrm{Lip}}
\newcommand\Dom{\mathrm{Dom}}
\newcommand\DDom{\widetilde{\mathrm{Dom}}}
\newcommand\HS{\mathrm{HS}}
\newcommand\law{\mathrm{law}}
\newcommand\lin{\mathrm{lin}}
\newcommand\loc{\mathrm{loc}}
\newcommand\op{\mathrm{op}}

\usepackage{url}

\usepackage{marginnote}

\begin{document}
\title[]{Optimal Wasserstein-$1$ distance between SDEs driven by Brownian motion and stable processes}



\author[C.-S.~Deng]{Chang-Song Deng}
\address[C.-S.~Deng]{School of Mathematics and Statistics\\ Wuhan University\\ Wuhan 430072, China}
\email{dengcs@whu.edu.cn}

\author[R.L.~Schilling]{Ren\'e L.\ Schilling}
\address[R.L.~Schilling]{TU Dresden\\ Fakult\"{a}t Mathematik\\ Institut f\"{u}r Mathematische Stochastik\\ 01062 Dresden, Germany}
\email{rene.schilling@tu-dresden.de}

\author[L.~Xu]{Lihu Xu}
\address[L.~Xu]{Department of Mathematics, Faculty of Science and Technology, University of Macau, Macau S.A.R., China }
\email{lihuxu@um.edu.mo}

\begin{abstract}
    We are interested in the following two $\rd$-valued stochastic differential equations (SDEs):
    \begin{gather*}
        \dup X_t=b(X_t)\,\dup t + \sigma\,\dup L_t, \quad X_0=x,\\
        \dup Y_t=b(Y_t)\,\dup t + \sigma\,\dup B_t, \quad Y_0=y,
    \end{gather*}
    where $\sigma$ is an invertible $d\times d$ matrix, $L_t$ is a rotationally symmetric $\alpha$-stable L\'evy process, and $B_t$ is a $d$-dimensional standard Brownian motion (note that $B_t$ is a rotationally symmetric $\alpha$-stable L\'evy process with $\alpha=2$). We show that for any $\alpha_0 \in (1,2)$ the Wasserstein-$1$ distance $W_1$ satisfies for $\alpha \in [\alpha_0,2)$
    \begin{gather*}
    W_{1}\left(\law (X_{t}^x), \law (Y_{t}^y)\right)
    \leq C_1\eup^{-C_2t}|x-y|
    +\frac{C}{\alpha_0-1}(2-\alpha)d\log(1+d),
    \end{gather*}
    which implies, in particular,
    \begin{equation} \label{e:W1Rate}
        W_1(\mu_\alpha, \mu_2)
        \leq \frac{C}{\alpha_0-1}(2-\alpha)d\log(1+d),
    \end{equation}
    where $\mu_\alpha$ and $\mu_2$ are the ergodic measures of $X_t$ and $Y_t$ respectively.

   For the special case of a $d$-dimensional Ornstein--Uhlenbeck system, we show that $W_1(\mu_\alpha, \mu_2) \geq C_{d}  (2-\alpha)$ for all $\alpha\in(1,2)$; this indicates that the convergence rate with respect to $\alpha$ in \eqref{e:W1Rate} is optimal.  The term $d\log(1+d)$ appearing in the estimate \eqref{e:W1Rate} seems to be optimal for the dimension $d$. 
\end{abstract}

\maketitle

\tableofcontents\marginnote

\noindent

\section{Introduction}

We study the following two $\rd$-valued stochastic differential equations (SDEs):
\begin{gather}\label{stableSDE}
    \dup X_t = b(X_t)\,\dup t + \sigma\,\dup L_t, \quad X_0=x,
\\\label{BM-SDE}
    \dup Y_t = b(Y_t)\,\dup t + \sigma\,\dup B_t, \quad Y_0=y,
\end{gather}
where $\sigma$ is an invertible $d\times d$ matrix, $B_t$ is a $d$-dimensional standard Brownian motion, and $L_t$ is a rotationally symmetric $\alpha$-stable L\'evy process with characteristic function $\Ee\,\eup^{\iup \xi L_t}=\eup^{-t|\xi|^\alpha/2}$. Under some suitable conditions, we can easily show that both equations have solutions which are ergodic.

SDEs driven by $\alpha$-stable processes have been intensively studied in recent years. We refer the reader to \cite{CHZ20, LSX20, Zha13, LW20} for gradient estimates, to \cite{Wan16, WXX17,PZ11, DSX23, DXZ14,Xu13, PSXY22} for structural properties and ergodicity,
and to \cite{CDSX22, CZZ21, KS19, JMW96} for the existence and uniqueness of solutions and approximation schemes.
The aim of this paper is to study the difference between the two ergodic measures in Wasserstein distance.

From L\'evy's continuity theorem, see e.g.\ \cite{Dur19}, we know that an $\alpha$-stable distribution converges to the normal distribution as $\alpha \uparrow 2$. It is natural and important to ask whether this convergence carries over to SDEs driven by Brownian motion and stable processes. As an application, one can justify that heavy tailed financial time series with second moment could be modeled by an SDE driven by Brownian motion \cite{Gro21}.  There have been several results in this direction, see for instance \cite{Liu22-2, Liu22, Liu22-1} and the references therein, but all of these results establish only the convergence without giving a rate.
In the present paper we obtain a convergence rate which is optimal.

Throughout the paper, we make the following two assumptions:

\begin{h1}
There exist constants $\theta_0>0$ and $K\geq 0$ such that
\begin{gather*}
    \langle x-y, b(x)-b(y)\rangle\leq
    -\theta_0|x-y|^2+K
    \quad\text{for all\ \ } x, y\in\rd;
\end{gather*}
\end{h1}
\begin{h2}
There exist constants $\theta_1,\theta_2,\theta_3\geq 0$ such that for all $x,v_1,v_2,v_3 \in \rd$,
\begin{align}
    |\nabla_{v_1} b(x)|
    &\leq\theta_1|v_1|, \label{H1-1'}\\
    |\nabla_{v_2}\nabla_{v_1} b(x)|
    &\leq\theta_2|v_1||v_2|, \label{H1-2}\\
    |\nabla_{v_3}\nabla_{v_2}\nabla_{v_1} b(x)|
    &\leq \theta_3|v_1||v_2||v_3|. \label{H1-3}
\end{align}
Here the directional derivatives will be defined more precisely later.
\end{h2}

It is well-known that if \eqref{H1-1'} holds, then both \eqref{stableSDE} and \eqref{BM-SDE}
have unique non-explosive (strong) solutions. Whenever we want to emphasize the starting point $X_0=x$ for a given $x \in \rd$,
 we will write $X_t^x$ instead of $X_t$; we use this also for $Y_t^y$ for a given $y \in \rd$.

\begin{notation}
We denote by $C(\rd,\real)$, $C^k(\rd,\real)$ the sets of continuous and $k$-times continuously differentiable functions; the subscripts ``$b$'' and ``$c$'' indicate that the functions and all their derivatives up to order $k$ are bounded, resp., have compact support.

Denote by $\Lip$ the set of all Lipschitz functions from $\rd$ to $\real$. The set of Lipschitz functions with Lipschitz constant $1$ is denoted by
\begin{gather*}
    \Lip(1)= \left\{h : \rd \to \real       \:;\: |h(x)-h(y)|\leq |x-y| \text{\ for all\ } x,y \in \rd \right\}.
\end{gather*}     
The Wasserstein-$1$ distance between two probability measures $\mu$ and $\nu$ is defined as
\begin{align*}
    W_1(\mu,\nu)
    &= \sup_{h \in \Lip(1)} \left\{\int h(x) \,\mu (\dif x) - \int h(x) \,\nu (\dif x)\right\} \\ 
    &= \sup_{h \in \Lip(1),|h(\cdot)| \leq |\cdot|} \left\{\int h(x) \,\mu (\dif x) - \int h(x) \,\nu (\dif x)\right\}.
\end{align*}
For $h\in\Lip(1)$, $\|\nabla h\|_\infty$ is defined as $\|\nabla h\|_\infty=\sup_{x,y \in \rd,x\neq y} \frac{|h(x)-h(y)|}{|x-y|}$, which is the Lipschitz constant.

Throughout this paper, $C,C_1,C_2$ denote positive constants which may depend on $\theta_0,\theta_1, \theta_2,\theta_3,K,
\|\sigma\|_{\op}$, but they are always independent of $d$ and $\alpha$; their values may change,
without further notice, from line to line.
Denote by $|x|$ the Euclidean norm of $x\in\rd$.
     
\end{notation}

\subsection{Main result}
From the classical Lyapunov function criterion \cite{MeTw09} or Harris' Theorem \cite{PSXZ12, Hai21} we know that the solutions to
the SDEs \eqref{stableSDE} and \eqref{BM-SDE} are ergodic. Denote by $\mu_\alpha$ and $\mu_2$ the respective ergodic
measures of $X_t^x$ and $Y_t^y$.
The following theorem is the main result of our paper.
\begin{theorem}\label{main1}
Assume that both \textup{\textbf{(H1)}} and \textup{\textbf{(H2)}} hold true, and let $\alpha_0\in(1,2)$ be an arbitrary number. For any $\alpha \in [\alpha_0,2)$, $x,y\in\rd$ and $t>0$, we have
\begin{gather*}
    W_{1}\left(\law (X_{t}^x), \law (Y_{t}^y)\right)
    \leq C_1\eup^{-C_2t}|x-y|
    +\frac{C}{\alpha_0-1}(2-\alpha)d\log(1+d).
\end{gather*}
In particular,
\begin{gather}\label{was-est}
    W_1(\mu_\alpha,\mu_2)
    \leq \frac{C}{\alpha_0-1}(2-\alpha)d\log(1+d),
\end{gather}
where $\mu_\alpha$ and $\mu_2$ are the ergodic
measures of $X_t^x$ and $Y_t^y$ respectively.
\end{theorem}

It seems to be difficult to improve the factor $d\log(1+d)$ in the estimate \eqref{was-est}. This becomes clear from the proof of Lemma~\ref{genediff}, since the estimate of the term $\mathsf{J}_{11}$ used in this lemma is sharp; further details are given shortly before the statement of Lemma~\ref{genediff}. For the particular case of a $d$-dimensional Ornstein--Uhlenbeck system,
we show in Section \ref{s:OU} that, with some constant $C_d>0$ depending only on $d$,
$W_1(\mu_\alpha, \mu_2) \geq C_{d}  (2-\alpha)$ for all $\alpha\in(1,2)$, from which we see that the convergence rate with respect to $\alpha$ in Theorem~\ref{main1} \eqref{was-est} is optimal.

Our method relies on Duhamel's principle and a comparison of the generators of the solutions of the two SDEs; it may be seen as
a continuous version of the probability approximation framework established in \cite{CSX22+}. Another
key ingredient of our analysis is Bismut's formula from Malliavin calculus.

\subsection{Preliminaries}
In order to prove the main result, we use the fact that the solutions $(X_t^x)_{t \geq 0}$ and $(Y_t^y)_{t \geq 0}$ to the SDEs
are Markov processes. The operator semigroup induced by the Markov process $(X_t^x)_{t\geq 0}$ is given by
\begin{gather*}
    P_t f(x) = \Ee f(X^x_t),\quad f \in C_b(\rd,\real),\; t>0.
\end{gather*}
The infinitesimal generator $\mathscr{A}^P$ is a closed operator defined on the set of continuous functions vanishing at infinity, $C_\infty = C_\infty(\rd,\real)$,
\begin{gather*}
    \Dom (\mathscr{A}^P)
    := \left\{f\in C_\infty \:;\:      g(x) =       \lim_{t\to 0} \frac{P_t f(x)-f(x)}{t} \text{\ \ exists\ for all $x$ and\ }     g     \in C_\infty\right\}, \\
    \mathscr{A}^P f(x) :=\lim_{t\to 0} \frac{P_t f(x)-f(x)}{t}.
\end{gather*}
It is well-known that $C_c^2(\rd,\real)\subset \Dom(\mathscr{A}^P)$.     

Similarly, we can consider the semigroup $Q_t f(x)=\Ee f(Y^x_t)$ associated to $(Y_t^y)_{t \geq 0}$ and its infinitesimal generator $\mathscr{A}^Q$.

In applications, we often study a semigroup acting on functions which do not belong to $C^\infty_c(\rd,\real)$ or $C_\infty(\rd,\real)$, and so we need to extend the domains $\Dom(\mathscr A^P)$ and $\Dom(\mathscr A^Q)$ to a larger function class.

Let $(X_t^x)_{t\geq 0}$ be the solution to \eqref{stableSDE}. Because of the Lipschitz property of $b$,  we have
\begin{align*}
    \Ee |X^x_t|
    &\leq |x| + \int_0^t (|b(0)| + C\Ee |X^x_s|)\, \dif s + \Ee |\sigma L_t|  \\
    &\leq |x| + |b(0)|t + C t^{1/\alpha} + C\int_0^t \Ee |X^x_s| \, \dif s,
\end{align*}
which we may combine with Gronwall's inequality to get
\begin{gather*}
    \Ee[|X^x_t|] \leq C_t(1+|x|)
    \quad\text{and, similarly, }\quad
    \Ee[|Y^x_t|] \leq C_t(1+|x|).
\end{gather*}
Thus, it is natural to consider the semigroups $(P_t)_{t \geq 0}$ and $(Q_t)_{t \geq 0}$ on the class of functions with linear growth:
\begin{gather*}     
    C_{\lin}(\rd,\real)
    = \left\{ f \in C(\rd,\real) \:;\: \sup_{x \in \rd}\frac{|f(x)|}{1+|x|}<\infty\right\}
\end{gather*}
and we define an extension of $\mathscr{A}^P$  as
\begin{gather*}
    \DDom (\mathscr{A}^P)
    := \left\{f\in C_{\lin}  \:;\:      g(x)=      \lim_{t\to 0} \frac{P_t f(x)-f(x)}{t} \text{\ \ exists \ for all $x$ and\ }     g     \in C_{\lin} \right\} \\
    \mathscr{A}^P f(x) :=\lim_{t\to 0} \frac{P_t f(x)-f(x)}{t}.
\end{gather*}
The argument in \cite[(1.7)]{Zha13} (this is also used in the proof of \eqref{grad-2} and \eqref{grad-3}) still holds for $f \in C_{\lin}(\rd,\real)$ and shows, in this case, that $|\nabla^k P_t f(x)| \le C(1+|x|) t^{-k/\alpha}$, $k=1,2$. Since $C^2_\lin(\rd,\real)$, the space of all twice differentiable functions from $\rd$ to $\real$ which grow, together with their deriviatives, at most linearly is contained in $\DDom(\mathscr{A}^P)$, we conclude that $P_t f \in \DDom(\mathscr{A}^P)$.
In a similar way we can extend $\mathscr{A}^Q$ onto $\DDom(\mathscr{A}^Q)$. Since we can approximate functions in $C^2_\lin(\rd,\real)$ locally uniformly with a sequence from $C_c^2(\rd,\real)$, it is clear that the Kolmogorov equations remain valid in a pointwise sense:

\noindent
The Kolmogorov backward equations: For all $f \in C_{\lin}(\rd,\real)$ and $t>0$,
\begin{gather*}
    \frac{\dif}{\dif t}\, P_t f
    = \mathscr{A}^P P_t f, \quad
    \frac{\dif}{\dif t}\, Q_t f=
    \mathscr{A}^Q Q_t f.
\intertext{The Kolmogorov forward equations: For all $t>0$,}
    \frac{\dif}{\dif t}\, P_t f=P_t \mathscr{A}^P  f,  \quad f \in \DDom (\mathscr{A}^P),\\
    \frac{\dif}{\dif t}\, Q_t f=Q_t \mathscr{A}^Q  f, \quad f \in \DDom (\mathscr{A}^Q).
\end{gather*}

For $f \in  C^3(\rd,\real)$ and $v_1, v_2,v_3,x \in \rd$, the directional derivatives  $\nabla_{v_1} f(x)$, $\nabla_{v_2}\nabla_{v_1} f(x)$ and $\nabla_{v_3}\nabla_{v_2}\nabla_{v_1} f(x)$ are  defined by
\begin{align*}
    \nabla_{v_1} f(x)&:=\lim_{\epsilon\to 0}
    \frac{f(x+\epsilon v_1)-f(x)}{\epsilon},  \\
    \nabla_{v_2} \nabla_{v_1} f(x)
    &:= \lim_{\epsilon\to 0}
    \frac{\nabla_{v_1}f(x+\epsilon v_2)-
    \nabla_{v_1}f(x)}{\epsilon},
\intertext{and}
    \nabla_{v_3}\nabla_{v_2} \nabla_{v_1} f(x)
    &:= \lim_{\epsilon\to 0}
    \frac{\nabla_{v_2}\nabla_{v_1}f(x+\epsilon v_3)-
    \nabla_{v_2}\nabla_{v_1}f(x)}{\epsilon}.
\end{align*}
For $f \in  C^2(\rd,\real)$, denote its gradient and Hessian
at $x\in\real^d$ as $\nabla f(x)$ and $\nabla^2f(x)$, respectively. For $v_1, v_2,x \in \rd$, it holds that
\begin{gather*}
	\nabla_{v_1}f(x)=\langle \nabla f(x),v_1\rangle,\quad
	\nabla_{v_2} \nabla_{v_1} f(x)=\langle \nabla^2f(x),v_1v_2^\top\rangle_\HS,
\end{gather*}
where     the superscript ``$\top$'' means ``transposition'',       and $\langle A,B\rangle_\HS:=\sum_{i,j=1}^dA_{ij}B_{ij}$ is the
Hilbert--Schmidt inner product of two matrices $A,B\in\real^{d\times d}$.
For any $x \in \real^d$, the operator norms of $\nabla^{2} f(x) \in\real^{d\times d}$
and $\nabla^{3} f(x) \in\real^{d\times d\times d}$ are given by
\begin{align*}
\|\nabla^{2} f(x)\|_{\op}
    &:=
    \sup\left\{|\nabla_{v_{2}} \nabla_{v_{1}} f(x)|\,;\,
    v_1,v_2\in\rd, |v_{1}|=|v_{2}|=1
    \right\}\\
    &=
    \sup\left\{|\langle \nabla^2f(x),v_1v_2^\top\rangle_\HS|\,;\,
    v_1,v_2\in\rd, |v_{1}|=|v_{2}|=1
    \right\},
\end{align*}
and
\begin{gather*}
    \|\nabla^3 f(x)\|_{\op}
    :=
    \sup\left\{|\nabla_{v_{3}}\nabla_{v_{2}} \nabla_{v_{1}} f(x)|\,;\,
    v_1,v_2,v_3\in\rd, |v_{1}|=|v_{2}|=|v_{3}|=1
    \right\},
\end{gather*}
respectively. For $f \in C_b^3(\rd,\real)$, we will use the supremum norms
\begin{gather*}
    \|\nabla f\|_{\infty}:=
    \sup_{x\in\rd}|\nabla f(x)|,
\qquad
    \|\nabla^if\|_{\mathrm{op},\infty}:=
    \sup_{x\in\rd}\|\nabla^if(x)\|_{\op}\quad (i=2,3).
\end{gather*}
{     
For notational simplicity, we also write $\nabla f(x)v_1=\nabla_{v_1}f(x)$,
$\nabla^2f(x)v_1v_2=\nabla_{v_2}\nabla_{v_1}f(x)$ and
$\nabla^3f(x)v_1v_2v_3=\nabla_{v_3}\nabla_{v_2}\nabla_{v_1}f(x)$ for $v_1,v_2,v_3\in\rd$.
}

{     
The directional derivatives are similarly defined for (sufficiently smooth) vector-valued
functions $f=(f_{1},f_{2},\cdots,f_{d})^{\top}: \rd\to \rd$:
for $v_1, v_2,v_3,x \in \rd$, then 
\begin{align*}
   \nabla f(x)v_1&=\nabla_{v_1}f(x)=(\nabla_{v_1}f_{1}(x),\dots,\nabla_{v_1}f_{d}(x))^{\top},\\
   \nabla^2f(x)v_1v_2&=\nabla_{v_{2}}\nabla_{v_{1}}f(x)= (\nabla_{v_{2}}\nabla_{v_{1}}f_{1}(x),\dots,\nabla_{v_{2}} \nabla_{v_{1}}f_{d}(x))^{\top},\\
   \nabla^3f(x)v_1v_2v_3&=\nabla_{v_{3}}\nabla_{v_{2}}\nabla_{v_{1}}f(x)= (\nabla_{v_{3}}\nabla_{v_{2}}\nabla_{v_{1}}f_{1}(x),
\dots,\nabla_{v_{3}}\nabla_{v_{2}} \nabla_{v_{1}}f_{d}(x))^{\top}.
\end{align*}
}

For a matrix $A\in\real^{d\times d}$, its Hilbert--Schmidt norm is $\|A\|_\HS=\sqrt{\langle A,A\rangle_\HS}$ and
its operator norm is $\|A\|_\op=\sup_{v\in\rd,|v|=1}|Av|$. We have the following relations:
\begin{gather*}
    \|A\|_\op=\sup_{v_1,v_2\in\rd,|v_1|=|v_2|=1} \left| \langle A,v_1v_2^\top\rangle_\HS\right|,\quad
    \|A\|_\op\leq\|A\|_\HS\leq\sqrt{d}\|A\|_\op.
\end{gather*}

For further use, set
\begin{equation} \label{e:AWDef}
    A(d,\alpha):=\frac{\alpha\Gamma(\frac{d+\alpha}{2})}
    {2^{2-\alpha}\pi^{d/2}\Gamma(1-\frac\alpha2)},\quad
    \omega_{d-1}
    :=\frac{2\pi^{d/2}}{\Gamma\left(\frac{d}{2}\right)}.
\end{equation}

\section{Proof of Theorem \ref{main1}}

For the proof of Theorem \ref{main1} we need some preparations. First, the following gradient estimates are crucial,
and     their       proofs are postponed to Section \ref{gradproof}.

\begin{lemma}\label{grad}
Assume \textup{\textbf{(H2)}}. Then for all $h \in \Lip(1)$ and $t\in(0,1]$,
\begin{align}\label{grad-1}
    \|\nabla Q_th\|_{\infty} &\leq C,
\\ \label{grad-2}
    \|\nabla^2 Q_th\|_{\op,\infty} &\leq Ct^{-1/2},
\\ \label{grad-3}
    \|\nabla^3 Q_th\|_{\op,\infty} &\leq C t^{-1}.
\end{align}
\end{lemma}

\begin{lemma}\label{crate}
    Let $A(d,\alpha)$ and $\omega_{d-1}$ be as in \eqref{e:AWDef}.  Then for all $\alpha\in(0,2)$,
    \begin{gather*}
        \frac{A(d,\alpha)\omega_{d-1}}{d(2-\alpha)}\leq C \quad \text{and} \quad
        \left| \frac{A(d,\alpha)\omega_{d-1}}{d(2-\alpha)} - 1\right|
        \leq C(2-\alpha)\log(1+d).
    \end{gather*}
\end{lemma}

\begin{proof}
Note that
\begin{gather*}
    \frac{A(d,\alpha)\omega_{d-1}}{d(2-\alpha)}
    = \frac{\alpha\Gamma\left(\frac{d+\alpha}{2}\right)}{d(2-\alpha)2^{1-\alpha}\Gamma\left(1-\frac\alpha2\right) \Gamma\left(\frac{d}{2}\right)}
    = \frac{\alpha\Gamma\left(\frac{d+\alpha}{2}\right)} {d2^{2-\alpha}\Gamma\left(2-\frac\alpha2\right) \Gamma\left(\frac{d}{2}\right)}.
\end{gather*}
Since for all $\alpha\in(0,2)$
\begin{gather*}
    \Gamma\left(\frac{d+\alpha}{2}\right)\leq C\Gamma\left(\frac{d+2}{2}\right)=C\,\frac{d}{2}\,\Gamma\left(\frac{d}{2}\right),
\intertext{we get}
    \frac{A(d,\alpha)\omega_{d-1}}{d(2-\alpha)}
    \leq\frac{C\alpha}{2^{3-\alpha}\Gamma\left(2-\frac\alpha2\right)}
    \leq C,
\end{gather*}
which     is just       the first estimate. To prove the second     estimate,       set
\begin{gather*}
    \rho(x)
    :=x\Gamma\left( \frac{d+x}{2} \right) - d2^{2-x}\Gamma\left(2-\frac x2\right) \Gamma\left(\frac{d}{2}\right),
    \quad 0\leq x\leq 2.
\end{gather*}
It is easy to see that $\rho(2)=0$ and
\begin{align*}
    \rho'(x)
    &= \Gamma\left(\frac{d+x}{2}\right) + \frac{x}{2}\,\Gamma'\left(\frac{d+x}{2}\right)\\
    &\quad \mbox{} + d\cdot\Gamma\left(\frac{d}{2}\right)
        \left\{ \log 2 \cdot 2^{2-x} \Gamma\left(2-\frac{x}{2}\right) + 2^{1-x}\Gamma'\left(2-\frac{x}{2}\right)\right\}\\
    &= \Gamma\left(\frac{d+x}{2}\right) + \frac{x}{2}\,\psi\left(\frac{d+x}{2}\right) \Gamma\left(\frac{d+x}{2}\right)\\
    &\quad \mbox{} + d\cdot\Gamma\left(\frac{d}{2}\right)
        \left\{\log 2\cdot 2^{2-x} \Gamma\left(2-\frac{x}{2}\right) + 2^{1-x}\psi\left(2-\frac{x}{2}\right)
        \Gamma\left(2-\frac{x}{2}\right)\right\},
\end{align*}
where $\psi(x)=\Gamma'(x)/\Gamma(x)$ is the Digamma function. Since $\lim_{z\to\infty}\psi(z)/\log z=1$, cf.\ \cite[6.3.18, p.~259]{abra-steg}, we have for all $x\in[0,2]$,
\begin{gather*}
        |\rho'(x)|\leq Cd\,\Gamma\left(\frac{d}{2}\right)\log(1+d).
\end{gather*}
From this we conclude that for all $\alpha\in(0,2)$
\begin{align*}
    \left| \frac{A(d,\alpha)\omega_{d-1}}{d(2-\alpha)} - 1\right|
    &= \frac{|\rho(\alpha)|}{d2^{2-\alpha}\Gamma\left(2-\frac\alpha2\right)\Gamma\left(\frac{d}{2}\right)}\\
    &= \frac{|\rho(2)-\rho(\alpha)|}{d2^{2-\alpha}\Gamma\left(2-\frac\alpha2\right)\Gamma\left(\frac{d}{2}\right)}\\
    &\leq \frac{C(2-\alpha)\log(1+d)}{2^{2-\alpha}\Gamma\left(2-\frac\alpha2\right)}\\
    &\leq C(2-\alpha)\log(1+d).
\qedhere
\end{align*}
\end{proof}

The next lemma is a key step in proving our main result. Note that the term $\mathsf{J}_{11}$, which appears in the proof below, has a sharp estimate, depending on $d\log (1+d)$. This seems to indicate that the bound in Theorem~\ref{main1} \eqref{was-est} cannot be improved.
\begin{lemma}\label{genediff}
    Assume that \textup{\textbf{(H2)}} holds and let $\alpha_0 \in (1,2)$ be an arbitrary fixed number.
    Then for any $\alpha \in [\alpha_0,2)$ and $s\in(0,1]$,
    \begin{align*}
        \sup_{h \in \Lip(1)}
        \left\|(\mathscr{A}^Q-\mathscr{A}^P)Q_sh\right\|_\infty
        \leq \frac{C}{\alpha_0-1}\left(
        (3-\alpha)s^{-1/2}-\frac{1}{3-\alpha}\,s^{(1-\alpha)/2}
        \right) d\log(1+d).
    \end{align*}
\end{lemma}

\begin{proof}
Set $f:=Q_sh$. Since
\begin{gather*}
    \mathscr{A}^Pf(x)
    = \langle\nabla f(x),b(x)\rangle + \int\limits_{\rd\setminus\{0\}}
    \left[f(x+\sigma z)-f(x)-\langle \nabla f(x),\sigma z\rangle\I_{\{|z|\leq1\}}\right]
    \frac{A(d,\alpha)}{|z|^{d+\alpha}}\,\dup z,
\intertext{and}
    \mathscr{A}^Qf(x)
    = \langle\nabla f(x),b(x)\rangle + \frac12\,\langle\nabla^2f(x),\sigma\sigma^{\top}\rangle_{\HS},
\end{gather*}
it follows that
\begin{align*}
    &(\mathscr{A}^P-\mathscr{A}^Q)f(x)\\
    &=\left\{ \int_{|z|\leq1}\left[f(x+\sigma z)-f(x) - \langle\nabla f(x),\sigma z\rangle\right] \frac{A(d,\alpha)}{|z|^{d+\alpha}}\,\dup z
    - \frac12\,\langle\nabla^2f(x),\sigma\sigma^{\top} \rangle_{\HS}\right\}\\
    &\quad\mbox{}+\int_{|z|>1}\left[f(x+\sigma z)-f(x)\right] \frac{A(d,\alpha)}{|z|^{d+\alpha}}\,\dup z\\
    &=:\mathsf{J}_1+\mathsf{J}_2.
\end{align*}
With the first inequality in Lemma \ref{crate} it is easy to check that for any $\alpha \in [\alpha_0,2)$ the following estimate holds:
\begin{gather*}
    \frac{A(d,\alpha)\omega_{d-1}}{\alpha-1}\leq \frac{C}{\alpha_0-1}(2-\alpha)d.
\end{gather*}
Because of \eqref{grad-1}, $f=P_sh \in \Lip$ if $h\in\Lip(1)$, and so for all $\alpha \in [\alpha_0,2)$,
\begin{align*}
    |\mathsf{J}_2|
    &\leq \int_{|z|>1}\left|f(x+\sigma z)-f(x)\right| \frac{A(d,\alpha)}{|z|^{d+\alpha}}\,\dup z\\
    &\leq C \int_{|z|>1}|z| \frac{A(d,\alpha)}{|z|^{d+\alpha}}\,\dup z\\
    &= C\frac{A(d,\alpha)\omega_{d-1}}{\alpha-1}\\
    &\leq \frac{C}{\alpha_0-1}(2-\alpha)d.
\end{align*}
We rewrite $\mathsf{J}_1$ in the following form:
\begin{align*}
    \mathsf{J}_1
    &= \int_{|z|\leq1}\left[f(x+\sigma z)-f(x) -\langle\nabla f(x),\sigma z\rangle\right] \frac{A(d,\alpha)}{|z|^{d+\alpha}}\,\dup z
    -\frac12\,\langle\nabla^2f(x),\sigma\sigma^{\top}\rangle_{\HS}\\
    &= \int_{|z|\leq1}\int_0^1\langle\nabla^2f(x+r\sigma z),(\sigma z)(\sigma z)^{\top}\rangle_{\HS}(1-r)\,\dup r \frac{A(d,\alpha)}{|z|^{d+\alpha}}\,\dup z
    - \frac12\,\langle\nabla^2f(x),\sigma\sigma^{\top}\rangle_{\HS}\\
    &= \left\{\int_{|z|\leq1}\int_0^1\langle\nabla^2f(x),(\sigma z)(\sigma z)^{\top}\rangle_{\HS}(1-r)\,\dup r \frac{A(d,\alpha)}{|z|^{d+\alpha}}\,\dup z
    -\frac12\,\langle\nabla^2f(x),\sigma\sigma^{\top}\rangle_{\HS}\right\}\\
    &\qquad\mbox{} + \int_{|z|\leq1}\int_0^1\langle\nabla^2f(x+r\sigma z)-\nabla^2f(x), (\sigma z)(\sigma z)^{\top}\rangle_{\HS}(1-r)\,\dup r\, \frac{A(d,\alpha)}{|z|^{d+\alpha}}\,\dup z\\
    &=: \mathsf{J}_{11}+\mathsf{J}_{12}.
\end{align*}
Using the symmetry of the measure $\varrho(\dif z) = |z|^{-d-\alpha}\dif z$, it is clear that $\int_{|z|\leq 1} z_i z_j\,\varrho(\dif z) = \delta_{ij} \frac 1d \int_{|z|\leq 1} |z|^2\,\varrho(\dif z)$, and so we get
\begin{align*}
    &\int_{|z|\leq1}\int_0^1\langle\nabla^2f(x),(\sigma z)(\sigma z)^{\top}\rangle_{\HS}(1-r)\,\dup r\frac{A(d,\alpha)}{|z|^{d+\alpha}}\,\dup z\\
    &\qquad= \frac12\int_{|z|\leq1}\langle\nabla^2f(x),\sigma(zz^{\top})\sigma^{\top}\rangle_{\HS}\frac{A(d,\alpha)}{|z|^{d+\alpha}}\,\dup z\\
    &\qquad=\frac12\,\frac{1}{d}\int_{|z|\leq1}\langle\nabla^2f(x),|z|^2\sigma\sigma^{\top}\rangle_{\HS}\frac{A(d,\alpha)}{|z|^{d+\alpha}}\,\dup z\\
    &\qquad=\frac{1}{2d}\,\langle\nabla^2f(x),\sigma\sigma^{\top}\rangle_{\HS}\int_{|z|\leq1}\frac{A(d,\alpha)}{|z|^{d+\alpha-2}}\,\dup z\\
    &\qquad=\frac{A(d,\alpha)\omega_{d-1}}{2d(2-\alpha)}\,\langle\nabla^2f(x),\sigma\sigma^{\top}\rangle_{\HS}.
\end{align*}
This, together with the second inequality in Lemma \ref{crate},
$\|\sigma\sigma^\top\|_\HS\leq\sqrt{d}\|\sigma\sigma^\top\|_\op\leq\sqrt{d}\|\sigma\|_\op^2$ and \eqref{grad-2}, implies
\begin{align*}
    |\mathsf{J}_{11}|
    &= \left|\int_{|z|\leq1}\int_0^1\langle\nabla^2f(x),(\sigma z)(\sigma z)^{\top}\rangle_{\HS}(1-r)\,\dup r \frac{A(d,\alpha)}{|z|^{d+\alpha}}\,\dup z
    -\frac12\,\langle\nabla^2f(x),\sigma\sigma^{\top}\rangle_{\HS}\right|\\
    &= \left|\frac12\left[\frac{A(d,\alpha)\omega_{d-1}}{d(2-\alpha)}-1\right]\langle\nabla^2f(x),\sigma\sigma^{\top}\rangle_{\HS}\right|\\
    &\leq C(2-\alpha)\log(1+d)\cdot\|\nabla^2f(x)\|_\HS\|\sigma\sigma^\top\|_\HS\\
    &\leq C(2-\alpha)\log(1+d)\cdot\sqrt{d}\|\nabla^2f(x)\|_\op\sqrt{d}\|\sigma\|_\op^2\\
    &\leq C(2-\alpha)d\log(1+d)\cdot\|\nabla^2f\|_{\op,\infty}\\
    &\leq C(2-\alpha)s^{-1/2}d\log(1+d).
\end{align*}
Now we turn to the estimate of $\mathsf{J}_{12}$. Since $f=P_sh$ and $h\in\Lip(1)$, we can use \eqref{grad-2} and \eqref{grad-3}
to see that for all $x\in\rd$,
\begin{align*}
    \left\|\nabla^2f(x+r\sigma z)-\nabla^2f(x)\right\|_\op
    &\leq C\left\{\left(\|\nabla^3f\|_{\op,\infty} r|\sigma z|\right)\wedge\|\nabla^2f\|_{\op,\infty}\right\}\\
    &\leq C\left\{\left(s^{-1}r|z|\right)\wedge s^{-1/2}\right\}.
\end{align*}
Using the first inequality in Lemma \ref{crate}, we have
\begin{gather*}
   \frac{A(d,\alpha)\omega_{d-1}}{2-\alpha}\leq C d.
\end{gather*}
Therefore,
\begin{align*}
    |\mathsf{J}_{12}|
    &\leq\int_{|z|\leq1}\int_0^1\left|\langle\nabla^2f(x+r\sigma z)-\nabla^2f(x),(\sigma z)(\sigma z)^{\top}\rangle_{\HS}\right|(1-r)\,\dup r\, \frac{A(d,\alpha)}{|z|^{d+\alpha}}\,\dup z\\
    &\leq\int_{|z|\leq1}\int_0^1\left\|\nabla^2f(x+r\sigma z)-\nabla^2f(x)\right\|_\op|\sigma z|^2
    (1-r)\,\dup r\, \frac{A(d,\alpha)}{|z|^{d+\alpha}}\,\dup z\\
    &\leq C \int_{|z|\leq1}\int_0^1\left\{ \left(s^{-1}r|z|\right)\wedge s^{-1/2}\right\}|z|^2(1-r)\,\dup r\, \frac{A(d,\alpha)}{|z|^{d+\alpha}}\,\dup z\\
    &\leq C \int_{|z|\leq1} \left\{\left(s^{-1}|z|\right)\wedge s^{-1/2}\right\} |z|^2 \frac{A(d,\alpha)}{|z|^{d+\alpha}}\,\dup z\\
    &=C\int_{|z|\leq s^{1/2}}s^{-1}|z|\frac{A(d,\alpha)}{|z|^{d+\alpha-2}}\,\dup z
    +C\int_{s^{1/2}<|z|\leq1}s^{-1/2}\frac{A(d,\alpha)}{|z|^{d+\alpha-2}}\,\dup z\\
    &=CA(d,\alpha)\omega_{d-1}\left(
    s^{-1}\int_0^{s^{1/2}}r^{2-\alpha}\,\dup r
    +s^{-1/2}\int_{s^{1/2}}^1r^{1-\alpha}\,\dup r
    \right)\\
    &=CA(d,\alpha)\omega_{d-1}\left(
    \frac{s^{(1-\alpha)/2}}{3-\alpha}+\frac{s^{-1/2}-s^{(1-\alpha)/2}}{2-\alpha}
    \right)\\
    &=C\frac{A(d,\alpha)\omega_{d-1}}{2-\alpha}\left(
    s^{-1/2}-\frac{1}{3-\alpha}\,s^{(1-\alpha)/2}
    \right)\\
    &\leq C\left(
    s^{-1/2}-\frac{1}{3-\alpha}\,s^{(1-\alpha)/2}
    \right)d.
\end{align*}
Combining all estimates, we get
for all $x\in\rd$ and $s\in(0,1]$,
\begin{align*}
    &\left|(\mathscr{A}^Q-\mathscr{A}^P)f(x)\right|\\
    &\leq |\mathsf{J}_{11}| + |\mathsf{J}_{12}|+|\mathsf{J}_{2}|\\
    &\leq \frac{C}{\alpha_0-1}(2-\alpha)s^{-1/2}d\log(1+d)
    +C\left(
    s^{-1/2}-\frac{1}{3-\alpha}\,s^{(1-\alpha)/2}
    \right)d\\
    &\leq \frac{C}{\alpha_0-1}(2-\alpha)s^{-1/2}d\log(1+d)
    +\frac{C}{\alpha_0-1}\left(
    s^{-1/2}-\frac{1}{3-\alpha}\,s^{(1-\alpha)/2}
    \right)d\log(1+d),
\end{align*}
which implies the claimed estimate.
\end{proof}

\begin{lemma}\label{geneint}
    Assume that \textup{\textbf{(H2)}} holds, and let $\alpha_0 \in (1,2)$ be an arbitrary fixed number.
    For all $t>0$ and $\alpha \in [\alpha_0,2)$,
    \begin{gather*}
        \sup_{h \in \Lip(1)} \left|\int_0^{t\wedge1}P_{t-s}(\mathscr{A}^Q-\mathscr{A}^P) Q_sh\,\dif s\right|
        \leq \frac{C}{\alpha_0-1}(2-\alpha)d\log(1+d).
    \end{gather*}
\end{lemma}

\begin{proof}
By Lemma \ref{genediff}, for all $h \in \Lip(1)$ and $\alpha \in [\alpha_0,2)$,
\begin{align*}
    &\left|\int_0^{t\wedge1}P_{t-s}(\mathscr{A}^Q-\mathscr{A}^P) Q_sh\,\dif s\right|\\
    &\qquad\leq \int_0^{t\wedge1} \left|P_{t-s}(\mathscr{A}^Q-\mathscr{A}^P) Q_sh\right| \dif s\\
    &\qquad\leq \int_0^{1} \left\|(\mathscr{A}^Q-\mathscr{A}^P) Q_sh\right\|_\infty \dif s\\
    &\qquad\leq \frac{C}{\alpha_0-1}\left(
    (3-\alpha)\int_0^1s^{-1/2}\,\dup s
    -\frac{1}{3-\alpha}\int_0^1s^{(1-\alpha)/2}\,\dup s
    \right)d\log(1+d)\\
    &\qquad=\frac{C}{\alpha_0-1}\,
    \frac{2(\alpha^2-7\alpha+13)(2-\alpha)}{(3-\alpha)^2}\,
    d\log(1+d)\\
    &\qquad\leq \frac{C}{\alpha_0-1}(2-\alpha)d\log(1+d).
    \qedhere
\end{align*}
\end{proof}

After these preparations we can now proceed with the proof of Theorem~\ref{main1}.

\begin{proof}[Proof of Theorem \ref{main1}]
(1) Note that
\begin{gather*}
     W_{1}\left(\law (X_{t}^x), \law (Y_{t}^y)\right)
     \leq
     W_{1}\left(\law (X_{t}^x), \law (Y_{t}^x)\right)
     +W_{1}\left(\law (Y_{t}^x), \law (Y_{t}^y)\right).
\end{gather*}
According to \cite[Corollary 2]{E16}, for any $t>0$ and $x,y\in \rd$,
    \begin{equation}\label{expcontr}
        W_{1}\left(\law (Y_{t}^x), \law (Y_{t}^y)\right)
        \le C_{1}e^{-C_{2}t}|x-y|.
    \end{equation}
Thus, it suffices to prove that for any $t>0$, $x\in \rd$ and $\alpha\in[\alpha_0,2)$,
\begin{equation}\label{samestarting}
    W_{1}\left(\law (X_{t}^x), \law (Y_{t}^x)\right)
    \leq \frac{C}{\alpha_0-1}(2-\alpha)d\log(1+d).
\end{equation}

Recall that
\begin{gather*}
    W_{1}\left(\law (X_{t}^x), \law (Y_{t}^x)\right)
    =\sup_{h\in\Lip(1)} |Q_th(x)-P_th(x)|
\end{gather*}
and use that
\begin{gather*}
    Q_t h-P_t h
    = \int_0^t \frac{\dif}{\dif s}\,P_{t-s} Q_s h \,\dif s
    =\int_0^t P_{t-s} (\mathscr{A}^Q-\mathscr{A}^P)Q_s h\,\dif s,
\end{gather*}
to get
\begin{equation}\label{dis11}
    W_{1}\left(\law (X_{t}^x), \law (Y_{t}^x)\right)
    = \sup_{h\in\Lip(1)} \left| \int_0^t P_{t-s} (\mathscr{A}^Q-\mathscr{A}^P)Q_s h(x) \,\dif s \right|.
\end{equation}
If $t\in(0,1]$, the desired estimate \eqref{samestarting} follows immediately from \eqref{dis11} and Lemma \ref{geneint}. We will now
consider the case $t>1$. By \eqref{dis11},
\begin{align*}
    W_{1}\left(\law (X_{t}^x), \law (Y_{t}^x)\right)
    &\leq \sup_{h\in\Lip(1)}\left| \int_0^1P_{t-s}(\mathscr{A}^Q-\mathscr{A}^P) Q_sh(x)\,\dif s \right|\\
    &\quad \mbox{} + \sup_{h\in\Lip(1)}\left|\int_1^tP_{t-s}(\mathscr{A}^Q-\mathscr{A}^P) Q_sh(x)\,\dif s \right|\\
    &=:\mathsf{I}_1+\mathsf{I}_2.
\end{align*}
Lemma \ref{geneint} shows that for all $\alpha\in[\alpha_0,2)$,
\begin{gather*}
    \mathsf{I}_1
    \leq \frac{C}{\alpha_0-1}(2-\alpha)d\log(1+d).
\end{gather*}
By \eqref{expcontr}, for $h\in\Lip(1)$ and $s>1$,
\begin{gather*}
    \left|Q_{s-1}h(x)-Q_{s-1}h(y)\right|
    \leq C_1\eup^{-C_2(s-1)}|x-y|,\quad x,y\in\rd.
\end{gather*}
Combining this with the semigroup property and Lemma \ref{genediff} with $s=1$, gives for all $x\in\rd$ and $\alpha\in[\alpha_0,2)$,
\begin{align*}
    \mathsf{I}_2
    &=\sup_{h\in\Lip(1)}\left| \int_1^tP_{t-s}(\mathscr{A}^Q-\mathscr{A}^P) Q_1(Q_{s-1}h)(x)\,\dif s \right|\\
    &\leq\sup_{g\in\Lip(1)}\left| \int_1^tC_1\eup^{-C_2(s-1)} P_{t-s}(\mathscr{A}^Q-\mathscr{A}^P) Q_1g(x)\,\dif s \right|\\
    &\leq\sup_{g\in\Lip(1)}\left\|(\mathscr{A}^Q-\mathscr{A}^P) Q_1g\right\|_\infty
    \cdot\int_1^tC_1\eup^{-C_2(s-1)} \,\dif s\\
    &\leq\frac{C}{\alpha_0-1}\,\frac{(4-\alpha)(2-\alpha)}{3-\alpha}\,d\log(1+d)\cdot\int_0^\infty C_1\eup^{-C_2s} \,\dif s\\
    &\leq\frac{C}{\alpha_0-1}(2-\alpha)d\log(1+d).
\end{align*}
Combining these two estimates implies \eqref{samestarting} for $t>1$. This proves the first assertion.

\medskip\noindent
(2)  From       the classical ergodic theory for Markov processes,  see e.g.\ \cite{M-T3},       it follows that for all $x\in\rd$
\begin{gather*}
    \lim_{t\to\infty}W_{1}\left(\law (X_{t}^x), \mu_\alpha\right)=0,
\end{gather*}
see e.g.\ \cite[(1.10)]{CDSX22} for details. Using        a similar argument, it is not hard to verify that
\begin{gather*}
    \lim_{t\to\infty}W_{1}\left(\law (Y_{t}^y), \mu_2\right)=0, \quad y\in\rd.
\end{gather*}
Since
\begin{gather*}
    W_{1}\left(\mu_\alpha,\mu_2\right)\leq W_{1}\left(\mu_\alpha,\law (X_{t}^x)\right)
    +W_{1}\left(\law (X_{t}^x),\law (Y_{t}^y)\right)
    +W_{1}\left(\law (Y_{t}^y),\mu_2\right),
\end{gather*}
the second assertion follows immediately from the first one with $t\rightarrow\infty$.
\end{proof}

\section{Proof of Lemma \ref{grad}}\label{gradproof}

\subsection{Gradient estimates for the SDE \eqref{BM-SDE}}

We consider the derivative of $Y_t^y$ with respect to the initial value $y\in\rd$. For $v\in\rd$, the directional derivative $\nabla_vY_t^y$ in direction $v$ is defined by
\begin{gather*}
    \nabla_vY_t^y=\lim_{\epsilon\to 0} \frac{Y_t^{y+\epsilon v}-Y_t^y}{\epsilon},
    \quad t\geq 0.
\end{gather*}
The above limit exists and satisfies the formally differentiated SDE
\begin{equation}\label{deriproc}
    \frac{\dif}{\dif t}\nabla_vY_t^y = \nabla b(Y_t^y)\nabla_vY_t^y,
    \quad \nabla_vY_0^y=v.
\end{equation}
In a similar way we can define for $v_i\in\rd$, $i=1,2,3$ the directional derivatives $\nabla_{v_2}\nabla_{v_1}Y_t^y$ and $\nabla_{v_3}\nabla_{v_2}\nabla_{v_1}Y_t^y$, which satisfy
\begin{gather*}
    \frac{\dif}{\dif t}\nabla_{v_2}\nabla_{v_1}Y_t^y
    =\nabla b(Y_t^y)\nabla_{v_2}\nabla_{v_1}Y_t^y + \nabla^2b(Y_t^y)\nabla_{v_2}Y_t^y\nabla_{v_1}Y_t^y
\end{gather*}
with $\nabla_{v_2}\nabla_{v_1}Y_0^y=0$, and
\begin{equation}\label{3rdd}
\begin{aligned}
    &\frac{\dif}{\dif t}\nabla_{v_3}\nabla_{v_2}\nabla_{v_1}Y_t^y
    =\nabla b(Y_t^y)\nabla_{v_3}\nabla_{v_2}\nabla_{v_1}Y_t^y
    +\nabla^3b(Y_t^y)\nabla_{v_3}Y_t^y\nabla_{v_2}Y_t^y
    \nabla_{v_1}Y_t^y\\
    &\qquad\mbox{}+\nabla^2b(Y_t^y)\left(
    \nabla_{v_3}Y_t^y\nabla_{v_2}\nabla_{v_1}Y_t^y
    +\nabla_{v_2}Y_t^y\nabla_{v_3}\nabla_{v_1}Y_t^y
    +\nabla_{v_1}Y_t^y\nabla_{v_3}\nabla_{v_2}Y_t^y
    \right),
\end{aligned}
\end{equation}
with $\nabla_{v_3}\nabla_{v_2}\nabla_{v_1}Y_0^y=0$.

\begin{lemma}\label{xgrad}
Assume \textup{\textbf{(H2)}}. Then for all $v_i\in\rd$, $i=1,2,3$, $y\in\rd$ and $t\in[0,1]$,
\begin{align}\label{gradesti-1}
    |\nabla_{v_1}Y_t^y|
    &\leq C |v_1|,
\\ \label{gradesti-2}
    |\nabla_{v_2}\nabla_{v_1}Y_t^y|
    &\leq C|v_1||v_2|,
\\ \label{gradesti-3}
    |\nabla_{v_3}\nabla_{v_2}\nabla_{v_1}Y_t^y|
    &\leq C|v_1||v_2||v_3|.
\end{align}
\end{lemma}

\begin{proof}
The estimates \eqref{gradesti-1} and \eqref{gradesti-2} can be proved as in \cite[Lemma 3.1]{CDSX22}.
Here we prove \eqref{gradesti-3} in a similar way. Set $\zeta(t):=\nabla_{v_3}\nabla_{v_2}\nabla_{v_1}Y_t^y$. We get from \eqref{3rdd}, \eqref{H1-1'}--\eqref{H1-3} and \eqref{gradesti-1}, \eqref{gradesti-2} that
\begin{align*}
    &\frac{\dif}{\dif t}|\zeta(t)|^2\\
    &= 2\left\langle\zeta(t),\nabla b(Y_t^y)\zeta(t)\right\rangle
    + 2\left\langle\zeta(t),\nabla^3 b(Y_t^y) \nabla_{v_3}Y_t^y\nabla_{v_2}Y_t^y\nabla_{v_1}Y_t^y \right\rangle
\\
    &\quad\mbox{}+2\left\langle\zeta(t),\nabla^2 b(Y_t^y)
        \left(
        \nabla_{v_3}Y_t^y\nabla_{v_2}\nabla_{v_1}Y_t^y
        +\nabla_{v_2}Y_t^y\nabla_{v_3}\nabla_{v_1}Y_t^y
        +\nabla_{v_1}Y_t^y\nabla_{v_3}\nabla_{v_2}Y_t^y
        \right)\right\rangle
\\
    &\leq 2\theta_1|\zeta(t)|^2+2\theta_3
        |\zeta(t)||\nabla_{v_1}Y_t^y||\nabla_{v_2}Y_t^y| |\nabla_{v_3}Y_t^y|
\\
    &\quad\mbox{}+2\theta_2|\zeta(t)|
        \left(|\nabla_{v_3}Y_t^y|
        |\nabla_{v_2}\nabla_{v_1}Y_t^y|
        +|\nabla_{v_2}Y_t^y|
        |\nabla_{v_3}\nabla_{v_1}Y_t^y|
        +|\nabla_{v_1}Y_t^y|
        |\nabla_{v_3}\nabla_{v_2}Y_t^y|
        \right)
\\
    &\leq2\theta_1|\zeta(t)|^2+C|\zeta(t)|
        |v_1||v_2||v_3|\\
        &\leq C|\zeta(t)|^2+C|v_1|^2|v_2|^2|v_3|^2.
\end{align*}
In the last line we use the elementary inequality $\pm 2z_1z_2\leq z_1^2 + z_2^2$. Noting that $\zeta(0)=0$, we get for $t>0$
\begin{gather*}
    |\zeta(t)|^2\leq C|v_1|^2|v_2|^2|v_3|^2 \int_0^t\eup^{C(t-r)}\,\dup r
    = \left(\eup^{Ct}-1\right) |v_1|^2|v_2|^2|v_3|^2.
\end{gather*}
This proves \eqref{gradesti-3} for all $t\in[0,1]$.
\end{proof}

\subsection{Bismut's formula}\label{malliavin}

Let $u\in L_{\loc}^{2}([0,\infty)\times(\Omega,\mathscr{F},\mathds{P});\rd)$, i.e.\ $\Ee\int_{0}^{t}|u(s)|^{2}\,\dif s<\infty$ for all $t>0$. Let $(B_{t})_{t\geq 0}$ be a standard Brownian motion on $\rd$ and assume that $u$ is adapted to the filtration $(\mathscr{F}_{t})_{t\geq 0}$ with $\mathscr{F}_{t}:=\sigma(B_{s}:0\leq s\leq t)$, and define $U:[0,\infty)\to\rd$ by
\begin{gather*}
    U_{t}:=\int_{0}^{t}u(s)\,\dif s,\quad t\geq 0.
\end{gather*}
For $t>0$, let $F_{t}:C([0,t],\rd)\to\real^m$ be an $\mathscr{F}_{t}$ measurable map, where $m\in\nat$. If the following limit exists
\begin{align*}
D_{U}F_{t}(B)=\lim_{\epsilon\to0}\frac{F_{t}(B+\epsilon U)-F_{t}(B)}{\epsilon}
\end{align*}
in $L^{2}((\Omega,\mathscr{F},\mathds{P});\real^m)$, then $F_{t}(B)$ is said to be \emph{Malliavin differentiable}, and $D_{U}F_{t}(B)$ is called the Malliavin derivative of $F_{t}(B)$ in the direction $U$.

Let both $F_{t}(B)$ and $G_{t}(B)$ be Malliavin differentiable. Then the following product rule holds:
\begin{gather*}
    D_{U}\left( F_{t}(B)G_{t}(B)\right)
    =  D_{U}F_{t}(B)G_{t}(B)  + F_{t}(B)D_{U}G_{t}(B).
\end{gather*}
If $F_t(B)$ has the following structure,
\begin{align*}
    F_{t}(B) = \int_{0}^{t}a_s\,\dif B_s,
\end{align*}
where $a_s\in\rd$ is an $\mathscr{F}_{s}$-adapted stochastic process such that $\mathds{E}\int_{0}^{t}|a_s|^{2}\,\dif s<\infty$ for all $t>0$, then
\begin{align}\label{chain}
    D_{U}F_{t}(B)
    = \int_{0}^{t}\langle a_s,u(s)\rangle\,\dif s + \int_{0}^{t} D_{U}a_s\,\dif B_{s}.
\end{align}
Note that $a_s$ is a $d$-dimensonal functional of Brownian motion $(B_r)_{0\leq r\leq s}$, and here we write
$D_{U}a_s=D_{U}a_s(B)$ for simplicity.

We will use the following integration by parts formula, which is also called \emph{Bismut's formula}. For any Malliavin differentiable $F_{t}(B)$ such that $F_{t}(B),D_{U}F_{t}(B)\in L^{2}((\Omega,\mathscr{F},\mathds{P});\real)$, we have
\begin{align}\label{bismut}
    \Ee[D_{U}F_{t}(B)]
    = \Ee \left[F_{t}(B)\int_{0}^{t} u(s)\,\dif B_{s}\right].
\end{align}

Let $\varphi\in \Lip(1)$ and let $F_{t}(B)$ be a $d$-dimensional Malliavin differentiable functional. Then the following chain rule holds:
\begin{align*}
    D_{U}\varphi(F_{t}(B))
    = \langle\nabla\varphi(F_{t}(B)),D_{U}F_{t}(B)\rangle.
\end{align*}

{     
If $F_{t}=(F_t^{(1)},\dots,F_t^{(d)}):C([0,t],\rd)\to\real^{d\times d}$, where $F_t^{(i)}\in\real^d$ 
is $\mathscr{F}_{t}$ measurable, then we can define $D_UF_t(B)=(D_UF_t^{(1)}(B),\dots,D_UF_t^{(d)}(B))$.
}

\subsection{Malliavin derivative estimates for the SDE \eqref{BM-SDE}}

For $t>0$ and $v_1,v_2,v_3,y\in\rd$, define $u_{t,i},U_{t,i}:[0,t] \to\rd$ by
\begin{gather*}
    u_{t,i}(s)
    := \frac{1}{t}\,\nabla_{v_i} Y_s^y
    \quad\text{and}\quad
    U_{t,i;s} := \int_0^su_{t,i}(r)\,\dif r
\end{gather*}
for $i=1,2,3$ and $s\in[0,t]$. Clearly,
\begin{align}\label{equality}
    D_{U_{t,i}}Y_s^y
    = \frac{s}{t} \,\nabla_{v_i}Y_s^y,
    \quad 0\leq s\leq t.
\end{align}
This, together with \eqref{deriproc}, implies that for $s\in[0,t]$
\begin{align}\label{SMalliavin}
    D_{U_{t,2}}\nabla_{v_{1}}Y_s^y
    = \int_0^{s}\left[\nabla^{2} b\left(Y_r^y\right)
    D_{U_{t,2}} Y_r^y\nabla_{v_{1}} Y_r^y +
    \nabla b\left(Y_r^y\right) D_{U_{t,2}}\nabla_{v_{1}}Y_r^y \right]\dif r.
\end{align}

The argument which we used in the proof of Lemma \ref{xgrad} gives the following upper bounds on the Malliavin derivatives.
\begin{lemma}
Assume \textup{\textbf{(H2)}}. For all $v_i\in\rd$, $i=1,2,3$, $y\in\rd$ and $t\in[0,1]$,
\begin{align}\label{dgradesti-1}
    |D_{U_{t,2}}\nabla_{v_1}Y_t^y|
    &\leq C|v_1||v_2|,
\\ \label{dgradesti-2}
    |D_{U_{t,3}} \nabla_{v_2}\nabla_{v_1}Y_t^y|
    &\leq C|v_1||v_2||v_3|,
\\ \label{dgradesti-3}
    |D_{U_{t,1}}D_{U_{t,3}}\nabla_{v_2}Y_t^y|
    &\leq C|v_1||v_2||v_3|.
\end{align}
\end{lemma}

\begin{proof}
From \eqref{SMalliavin} and \eqref{equality} we see that
\begin{gather*}
    \frac{\dif}{\dif t}\,D_{U_{t,2}} \nabla_{v_1}Y_t^y
    = \nabla b(Y_t^y) D_{U_{t,2}} \nabla_{v_1}Y_t^y
        +\nabla^2 b(Y_t^y)
        \nabla_{v_2}Y_t^y
        \nabla_{v_1}Y_t^y.
\end{gather*}
Repeating the argument used in the proof of Lemma \ref{xgrad}, we get \eqref{dgradesti-1} for all $t\in[0,t]$.
The estimates \eqref{dgradesti-2} and \eqref{dgradesti-3} can be proved in a similar way.
\end{proof}

\subsection{Proof of Lemma \ref{grad}}

We will frequently use the following mollifier: Let $g_\delta$ be the density of the $d$-dimensional normal distribution $N(0,\delta^{2}I_{d})$,  $\delta>0$. We define for every $h\in\Lip(1)$
\begin{gather*}
    h_\delta(x)
    :=\int_{\rd} g_\delta(y) h(x-y)\,\dif y.
\end{gather*}
It is easy to see that $h_\delta$ is smooth, $\lim_{\delta\downarrow0}h_\delta(x)=h(x)$ for all $x\in\rd$, and
\begin{gather*}
    \|\nabla h_\delta\|_{\infty}
    \leq \|\nabla h\|_{\infty}\leq 1.
\end{gather*}

\subsubsection{Proof of \eqref{grad-1}}\label{subsec:pf-grad-1}
The dominated convergence theorem
and \eqref{gradesti-1} imply for any $v\in\rd$  and $t\in[0,1]$
\begin{gather*}
    |\nabla_vQ_th_\delta(y)|
    = |\Ee[\nabla h_\delta(Y_t^y) \nabla_vY_t^y]|
    \leq \|\nabla h_\delta\|_{\infty} \Ee|\nabla_vY_t^y|
    \leq C|v|.
\end{gather*}
Since it holds from the dominated convergence theorem that
\begin{gather*}
    \lim_{\delta\downarrow0}\nabla_vQ_th_\delta(y)=\nabla_vQ_th(y),
\end{gather*}
the desired estimate follows by letting $\delta\downarrow0$.

\subsubsection{Proof of \eqref{grad-2}}\label{subsec:pf-grad-2}

We use the dominated convergence theorem to see that for all $v_1,v_2,y\in\rd$
\begin{equation}\label{2ndestidelta}
    \nabla_{v_{2}}\nabla_{v_{1}}\Ee\left[h_\delta(Y_t^y)\right]
    = \Ee\left[\nabla h_\delta(Y_t^y)\nabla_{v_{2}}\nabla_{v_{1}}Y_t^y\right]
    + \Ee\left[\nabla^{2}h_\delta(Y_t^y)\nabla_{v_{2}}Y_t^y \nabla_{v_{1}}Y_t^y\right].
\end{equation}
Now we use \eqref{gradesti-2} to get for $t\in[0,1]$,
\begin{equation}\label{1stestidelta}
    \left|\Ee\left[\nabla h_\delta(Y_t^y)\nabla_{v_{2}} \nabla_{v_{1}}Y_t^y\right]\right|
    \leq \|\nabla h_\delta\|_{\infty} \Ee\left[\left| \nabla_{v_{2}} \nabla_{v_{1}}Y_t^y \right|\right]
    \leq C|v_{1}||v_{2}|.
\end{equation}
From \eqref{equality} and \eqref{bismut} we get
\begin{align*}
    &\Ee\left[\nabla^{2}h_\delta(Y_t^y) \nabla_{v_{2}}Y_t^y \nabla_{v_{1}}Y_t^y\right]\\
    &\quad= \Ee\left[\nabla^{2}h_\delta(Y_t^y) D_{U_{t,2}}Y_t^y \nabla_{v_{1}}Y_t^y\right]\\
    &\quad= \Ee\left[D_{U_{t,2}}\left(\nabla h_\delta(Y_t^y)\right) \nabla_{v_{1}}Y_t^y\right]\\
    &\quad= \Ee\left[D_{U_{t,2}}\left(\nabla h_\delta(Y_t^y) \nabla_{v_{1}}Y_t^y\right)\right]
    -\Ee\left[\nabla h_\delta(Y_t^y) D_{U_{t,2}} \nabla_{v_{1}}Y_t^y \right]\\
    &\quad= \frac{1}{t}\,\Ee\left[ \nabla h_\delta(Y_t^y) \nabla_{v_{1}}Y_t^y
    \int_0^t\nabla_{v_{2}}Y_s^y\, \dif B_s \right]
    -\Ee\left[\nabla h_\delta(Y_t^y) D_{U_{t,2}} \nabla_{v_{1}}Y_t^y\right].
\end{align*}
Combining this with \eqref{gradesti-1} and \eqref{dgradesti-1}, It\^{o}'s isometry shows
for $t\in(0,1]$
\begin{align*}
    &\left|\Ee\left[\nabla^{2}h_\delta(Y_t^y) \nabla_{v_{2}}Y_t^y \nabla_{v_{1}}Y_t^y\right]\right|\\
    &\leq \frac{1}{t}\, \|\nabla h_\delta\|_{\infty}
    \Ee\left[\left|\nabla_{v_{1}}Y_t^y\right| \left|
    \int_0^t\nabla_{v_{2}}Y_s^y\, \dif B_s\right|\right]
    + \|\nabla h_\delta\|_{\infty} \Ee\left[\left| D_{U_{t,2}} \nabla_{v_{1}}Y_t^y\right| \right]\\
    &\leq \frac{C}{t}\, |v_1| \left[\Ee \left|\int_0^t\nabla_{v_{2}}Y_s^y\, \dif B_s\right|^2\right]^{1/2}
    + C|v_1||v_2| \\
    &\leq \frac{C}{t}\,|v_1||v_2| \sqrt{t} + C|v_1||v_2|\\
    &\leq C|v_1||v_2| t^{-1/2}.
\end{align*}    
Inserting this and \eqref{1stestidelta} into \eqref{2ndestidelta}, we get for $t\in (0,1])$,     
\begin{align*}
    \left|\nabla_{v_{2}}\nabla_{v_{1}} \Ee[h_\delta(Y_t^y)]\right|
    &\leq \left|\Ee[\nabla h_\delta(Y_t^y)\nabla_{v_{2}}\nabla_{v_{1}}Y_t^y]\right|
    + \left|\Ee\left[\nabla^{2}h_\delta(Y_t^y) \nabla_{v_{2}}Y_t^y \nabla_{v_{1}}Y_t^y\right]\right|\\
    &\leq C|v_1||v_2| t^{-1/2}.
\end{align*}
We can now let $\delta\downarrow0$ and we get with the help of the dominated convergence theorem
\begin{gather*}
    \lim_{\delta\downarrow0}\nabla_{v_{2}}\nabla_{v_{1}} \Ee[h_\delta(Y_t^y)]
    = \nabla_{v_{2}}\nabla_{v_{1}} \Ee[h(Y_t^y)]
    = \nabla_{v_{2}}\nabla_{v_{1}}Q_th(y).
\end{gather*}
Therefore, we get for all $v_1,v_2,y\in\rd$ and $t\in(0,1]$,
\begin{gather*}
    |\nabla_{v_{2}}\nabla_{v_{1}}Q_th(y)|
    \leq C|v_1||v_2|
    t^{-1/2},
\end{gather*}
and this implies \eqref{grad-2}.

\subsubsection{Proof of \eqref{grad-3}}\label{subsec:pf-grad-3}

We use the dominated convergence theorem to see that for all $v_1,v_2,v_3\in\rd$,
\begin{align*}
    &\nabla_{v_{3}}\nabla_{v_{2}}\nabla_{v_{1}}\Ee\left[h_\delta(Y_t^y)\right]\\
    &\quad=\Ee\left[\nabla h_\delta(Y_t^y)\nabla_{v_{3}} \nabla_{v_{2}}\nabla_{v_{1}}Y_t^y\right]
        +\Ee\left[\nabla^{2}h_\delta(Y_t^y)\nabla_{v_{3}}Y_t^y \nabla_{v_{2}}\nabla_{v_{1}}Y_t^y\right]\\
    &\qquad\mbox{} +\Ee\left[\nabla^{2}h_\delta(Y_t^y)\nabla_{v_{2}}Y_t^y \nabla_{v_{3}}\nabla_{v_{1}}Y_t^y\right]
        +\Ee\left[\nabla^{2}h_\delta(Y_t^y)\nabla_{v_{1}}Y_t^y \nabla_{v_{3}}\nabla_{v_{2}}Y_t^y\right]\\
    &\qquad\mbox{}+\Ee\left[\nabla^{3}h_\delta(Y_t^y)\nabla_{v_{3}}Y_t^y \nabla_{v_{2}}Y_t^y\nabla_{v_{1}}Y_t^y\right]\\
    &\quad=:\mathsf{I}_1+\mathsf{I}_2+\mathsf{I}_3+\mathsf{I}_4+\mathsf{I}_5.
\end{align*}
We will estimate the terms $\mathsf{I}_k$, $k=1,\dots,5$ separately. With \eqref{gradesti-3} we get for $t\in(0,1]$,
\begin{gather*}
    |\mathsf{I}_1|
    \leq \|\nabla h_\delta\|_{\infty} \Ee\left[\left| \nabla_{v_{3}}\nabla_{v_{2}} \nabla_{v_{1}}Y_t^y \right|\right]
    \leq C|v_{1}||v_{2}||v_{3}|.
\end{gather*}
Now we turn to $\mathsf{I}_2$. By \eqref{equality}     (with $s=t$), the chain rule       and \eqref{bismut},
\begin{align*}
    \mathsf{I}_2
    &= \Ee\left[\nabla^{2}h_\delta(Y_t^y) D_{U_{t,3}}Y_t^y \nabla_{v_{2}}\nabla_{v_{1}}Y_t^y\right]\\
    &= \Ee\left[D_{U_{t,3}}\left(\nabla h_\delta(Y_t^y)\right) \nabla_{v_{2}}\nabla_{v_{1}}Y_t^y\right]\\
    &= \Ee\left[D_{U_{t,3}}\left(\nabla h_\delta(Y_t^y) \nabla_{v_{2}}\nabla_{v_{1}}Y_t^y \right)\right]
        - \Ee\left[\nabla h_\delta(Y_t^y) D_{U_{t,3}} \nabla_{v_{2}}\nabla_{v_{1}}Y_t^y \right]\\
    &= \frac{1}{t}\,\Ee\left[ \nabla h_\delta(Y_t^y) \nabla_{v_{2}}\nabla_{v_{1}}Y_t^y
        \int_0^t\nabla_{v_{3}}Y_s^y\, \dif B_s \right]
        - \Ee\left[\nabla h_\delta(Y_t^y) D_{U_{t,3}} \nabla_{v_{2}}\nabla_{v_{1}}Y_t^y \right].
\end{align*}
Combining this with \eqref{gradesti-1}, \eqref{gradesti-2},  \eqref{dgradesti-2} and It\^{o}'s isometry,
we get for $t\in(0,1]$,
\begin{align*}
    |\mathsf{I}_2|
    &\leq \frac{C}{t}\, \|\nabla h_\delta\|_{\infty} |v_1||v_2|
        \left[\Ee \left| \int_0^t\nabla_{v_{3}}Y_t^y\, \dif B_s\right|^2\right]^{1/2}
        +\|\nabla h_\delta\|_{\infty} \Ee\left[\left| D_{U_{t,3}} \nabla_{v_{2}}\nabla_{v_{1}}Y_t^y\right| \right]\\
    &\leq \frac{C}{t}\,|v_1||v_2||v_3| \sqrt{t} + C|v_1||v_2||v_3|\\
    &\leq C|v_1||v_2||v_3| t^{-1/2}.
\end{align*}
In a similar way we can show that
\begin{gather*}
    |\mathsf{I}_3| + |\mathsf{I}_4|  \leq C|v_1||v_2||v_3|
    t^{-1/2},
    \quad t\in(0,1].
\end{gather*}
We still have to estimate $|\mathsf{I}_5|$. By \eqref{equality}     (with $s=t$), the chain rule       and \eqref{bismut},
\begin{align*}
    \mathsf{I}_5
    &= \Ee\left[\nabla^{3}h_\delta(Y_t^y)
    D_{U_{t,3}}Y_t^y
    \nabla_{v_{2}}Y_t^y
    \nabla_{v_{1}}Y_t^y\right]\\
    &= \Ee\left[D_{U_{t,3}}\left(\nabla^2 h_\delta(Y_t^y)\right)
    \nabla_{v_{2}}Y_t^y
    \nabla_{v_{1}}Y_t^y\right]\\
    &= \Ee\left[D_{U_{t,3}}\left(\nabla^2 h_\delta(Y_t^y)
    \nabla_{v_{2}}Y_t^y
    \nabla_{v_{1}}Y_t^y
    \right)\right]\\
    &\quad\mbox{}-\Ee\left[\nabla^2 h_\delta(Y_t^y)
    \nabla_{v_{1}}Y_t^y D_{U_{t,3}}
    \nabla_{v_{2}}Y_t^y\right]-\Ee\left[\nabla^2 h_\delta(Y_t^y)
    \nabla_{v_{2}}Y_t^yD_{U_{t,3}}
    \nabla_{v_{1}}Y_t^y\right]\\
    &= \frac{1}{t}\,
    \Ee\left[\nabla^2 h_\delta(Y_t^y)
    \nabla_{v_{2}}Y_t^y
    \nabla_{v_{1}}Y_t^y
    \int_0^t\nabla_{v_{3}}Y_s^y\,\dif B_s\right]\\
    &\quad\mbox{}-\Ee\left[\nabla^2 h_\delta(Y_t^y)
    \nabla_{v_{1}}Y_t^y D_{U_{t,3}}
    \nabla_{v_{2}}Y_t^y\right]-\Ee\left[\nabla^2 h_\delta(Y_t^y)
    \nabla_{v_{2}}Y_t^y D_{U_{t,3}}
    \nabla_{v_{1}}Y_t^y\right]\\
    & =: \mathsf{I}_{51}+\mathsf{I}_{52}+\mathsf{I}_{53}.
\end{align*}
For $\mathsf{I}_{51}$, it follows from \eqref{equality} and \eqref{chain} that
\begin{align*}
    \mathsf{I}_{51}
    &= \frac{1}{t}\,
    \Ee\left[D_{U_{t,2}}\left(\nabla h_\delta(Y_t^y)\right)
    \nabla_{v_{1}}Y_t^y \int_0^t\nabla_{v_{3}}Y_t^y\,\dif B_s\right]\\
    &= \frac{1}{t}\,\Ee\left[D_{U_{t,2}}\left(\nabla h_\delta(Y_t^y) \nabla_{v_{1}}Y_t^y
    \int_0^t\nabla_{v_{3}}Y_s^y\,\dif B_s\right)\right]\\
    &\qquad\mbox{}-\frac{1}{t}\, \Ee\left[\nabla h_\delta(Y_t^y)
    D_{U_{t,2}}\nabla_{v_{1}}Y_t^y\int_0^t\nabla_{v_{3}}Y_s^y\,\dif B_s\right]\\
    &\qquad\mbox{}-\frac{1}{t}\,\Ee\left[\nabla h_\delta(Y_t^y)
    \nabla_{v_{1}}Y_t^y \int_0^tD_{U_{t,2}}\nabla_{v_{3}}Y_s^y\,\dif B_s\right]\\
    &\qquad\mbox{}-\frac{1}{t^2}\,\Ee\left[\nabla h_\delta(Y_t^y)
    \nabla_{v_{1}}Y_t^y
    \int_0^t \nabla_{v_{3}}Y_s^y \nabla_{v_{2}}Y_s^y\,\dif s\right]\\
    &=: \mathsf{I}_{511}+\mathsf{I}_{512}+\mathsf{I}_{513}+\mathsf{I}_{514}.
\end{align*}
By  \eqref{bismut}, \eqref{gradesti-1} and It\^{o}'s isometry, we get for $t\in(0,1]$,
\begin{align*}
    \left|\mathsf{I}_{511}\right|
    &=\left|\frac{1}{t^2}\, \Ee\left[\nabla h_\delta(Y_t^y)
    \nabla_{v_{1}}Y_t^y
    \int_0^t\nabla_{v_{3}}Y_s^y \,\dif B_s
    \int_0^t\nabla_{v_{2}}Y_s^y \,\dif B_s \right]\right| \\
    &\leq \frac{C}{t^2}\, \|\nabla h_\delta\|_{\infty}  |v_1|
    \left(\Ee\left| \int_0^t\nabla_{v_{3}}Y_s^y \,\dif B_s\right|^2\right)^{1/2}
    \left(\Ee\left| \int_0^t\nabla_{v_{2}}Y_s^y \,\dif B_s\right|^2\right)^{1/2} \\
    &\leq C|v_1||v_2||v_3|t^{-1}.
\end{align*}
By \eqref{dgradesti-1}, It\^{o}'s isometry and \eqref{gradesti-1}, we get for $t\in(0,1]$,
\begin{align*}
    \left|\mathsf{I}_{512}\right|
    &\leq \frac{1}{t}\,\|\nabla h_\delta\|_{\infty} |v_1| |v_2|
    \left(\Ee\left| \int_0^t\nabla_{v_{3}}Y_s^y \,\dif B_s\right|^2\right)^{1/2} \\
    &\leq C|v_1||v_2||v_3|t^{-1/2}
\end{align*}
and, with a similar calculation,
\begin{align*}
    \left|\mathsf{I}_{513}\right|
    \leq C|v_1||v_2||v_3|t^{-1/2}.
\end{align*}
Using \eqref{gradesti-1} we see for $t\in(0,1]$ that
\begin{gather*}
    \left|\mathsf{I}_{514}\right|
    \leq \frac{C}{t^2}\,\|\nabla h_\delta\|_\infty|v_1| |v_2| |v_3|t
    \leq C|v_1||v_2||v_3|t^{-1}.
\end{gather*}
If we combine the estimates for $|\mathsf{I}_{511}|,\dots,|\mathsf{I}_{514}|$, we obtain for $t\in(0,1]$,
\begin{align*}
    \left|\mathsf{I}_{51}\right|
    \leq C|v_1||v_2||v_3|\left(t^{-1}+t^{-1/2}\right)
    \leq C|v_1||v_2||v_3|t^{-1}.
\end{align*}
    Using again       \eqref{equality} and \eqref{bismut},     we see       that
 \begin{align*}
    \mathsf{I}_{52}
    &= -\Ee\left[\nabla^{2}h_\delta(Y_t^y) D_{U_{t,1}}Y_t^yD_{U_{t,3}} \nabla_{v_{2}}Y_t^y\right]\\
    &= -\Ee\left[D_{U_{t,1}}\left(\nabla h_\delta(Y_t^y)\right) D_{U_{t,3}}\nabla_{v_{2}}Y_t^y\right]\\
    &= -\Ee\left[D_{U_{t,1}}\left(\nabla h_\delta(Y_t^y)D_{U_{t,3}} \nabla_{v_{2}}Y_t^y \right)\right]
        +\Ee\left[\nabla h_\delta(Y_t^y)D_{U_{t,1}} D_{U_{t,3}} \nabla_{v_{2}}Y_t^y \right]\\
    &= -\frac{1}{t}\,\Ee\left[ \nabla h_\delta(Y_t^y)D_{U_{t,3}} \nabla_{v_{2}}Y_t^y
        \int_0^t\nabla_{v_{1}}Y_s^y\, \dif B_s \right]
        + \Ee\left[\nabla h_\delta(Y_t^y)D_{U_{t,1}} D_{U_{t,3}} \nabla_{v_{2}}Y_t^y \right].
\end{align*}   
Now       we can use \eqref{dgradesti-1}, It\^{o}'s isometry, \eqref{gradesti-1} and \eqref{dgradesti-3} to get
for $t\in(0,1]$,
\begin{align*}
    |\mathsf{I}_{52}|
    &\leq \frac{C}{t}\, \|\nabla h_\delta\|_{\infty}
    |v_2||v_3|
    \left(\Ee\left| \int_0^t\nabla_{v_{1}}Y_s^y \,\dif B_s\right|^2\right)^{1/2}\\
    &\qquad\mbox{} +\|\nabla h_\delta\|_{\infty}
    \Ee\left[\left| D_{U_{t,1}}D_{U_{t,3}} \nabla_{v_{2}}Y_t^y\right| \right]\\
    &\leq \frac{C}{t}\,|v_1||v_2||v_3|\sqrt{t}+C|v_1| |v_2| |v_3|\\
    &\leq C|v_1| |v_2| |v_3|t^{-1/2}.
\end{align*}
The same argument gives
\begin{align*}
    |\mathsf{I}_{53}| \leq C|v_1| |v_2| |v_3|t^{-1/2},
    \quad t \in (0,1].
\end{align*}
If we combine the estimates for $|\mathsf{I}_{51}|, |\mathsf{I}_{52}|, |\mathsf{I}_{53}|$, we obtain
for all $t\in (0,1]$
\begin{gather*}
    |\mathsf{I}_{5}|\leq C|v_1| |v_2| |v_3|\left(t^{-1}+t^{-1/2}\right)
    \leq C|v_1| |v_2| |v_3| t^{-1}.
\end{gather*}
Finally we can combine the bounds for $|\mathsf{I}_i|$, $i=1,2,3,4,5$, and we get for $t\in(0,1]$,
\begin{align*}
    \left|\nabla_{v_{3}}\nabla_{v_{2}}\nabla_{v_{1}} \Ee\left[h_\delta(Y_{t}^{y})\right]\right|
    &\leq |\mathsf{I}_1| + \dots + |\mathsf{I}_5|\\
    &\leq C|v_1||v_2||v_3|\left(1+t^{-1/2}+t^{-1}\right)\\
    &\leq C|v_1||v_2||v_3| t^{-1}.
\end{align*}
We can now let $\delta\downarrow0$ and use dominated convergence to see for all $v_1,v_2,v_3,y\in\rd$ and $t\in(0,1]$,
\begin{gather*}
    \left|\nabla_{v_{3}}\nabla_{v_{2}}\nabla_{v_{1}} Q_th(y)\right|
    =
    \lim_{\delta\downarrow 0} \left|\nabla_{v_{3}}\nabla_{v_{2}}\nabla_{v_{1}} \Ee\left[h_\delta(Y_{t}^{y})\right]\right|
    \leq
    C|v_1||v_2||v_3| t^{-1},
\end{gather*}
which yields \eqref{grad-3}. This completes the proof.

\section{A lower bound for the Ornstein--Uhlenbeck case} \label{s:OU}

In this section we establish a lower bound for the Ornstein--Uhlenbeck case, i.e.\ for $b(x)=-x$ and $\sigma=I_{d\times d}$. In this case, \textup{\textbf{(H1)}} holds with $\theta_0=1$ and $K=0$, and \textup{\textbf{(H2)}} holds with $\theta_1=1$ and $\theta_2=\theta_3=0$.
In this section, $\mu_\alpha$ and $\mu_2$ are the ergodic measures of the solutions to the SDEs
\begin{gather*}\marginnote{    One line}
        \dup X_t=-X_t\,\dup t + \dup L_t
        \quad\text{resp.}\quad
        \dup Y_t=-Y_t\,\dup t + \dup B_t.
\end{gather*}

\begin{proposition}
There exists a constant $C_d>0$ depending only on $d$ such that for any $\alpha\in(1,2)$,
\begin{gather*}
    W_1(\mu_\alpha,\mu_2)\geq C_{d}(2-\alpha).
\end{gather*}
\end{proposition}

\begin{proof}
Since $X_{t}=\eup^{-t}x+\eup^{-t}\int_{0}^{t}\eup^{s}\,\dif L_{s}$, we get for $\xi\in\rd$,
\begin{align*}
    \Ee\left[\eup^{\iup\xi X_{t}}\right]
    &= \eup^{\iup\xi\eup^{-t}x} \Ee\left[ \eup^{\iup\int_{0}^{t}\xi \eup^{-t}\eup^{s}\,\dif L_{s}}\right]
    = \eup^{\iup\xi\eup^{-t}x}\eup^{-2^{-1}\int_{0}^{t}|\xi \eup^{-t}\eup^{s}|^{\alpha}\,\dif s}\\
    &= \eup^{\iup\xi\eup^{-t}x} \eup^{-(2\alpha)^{-1}|\xi|^{\alpha}(1-\eup^{-\alpha t})}
    \;\xrightarrow[]{\; t\to\infty\; }\;
    \eup^{-(2\alpha)^{-1}|\xi|^{\alpha}}
    =\Ee\left[\eup^{\iup\xi\alpha^{-1/\alpha}L_{1}}\right].
\end{align*}
Thus, the ergodic measure $\mu_\alpha$ is given by the law of $\alpha^{-1/\alpha}L_{1}$. Similarly, the ergodic measure $\mu_2$ is given by the law of $2^{-1/2}B_{1}$. We know from \cite[Lemma 4.2]{DS19} that
\begin{gather*}
    \Ee\left|L_{1}\right|
    =\frac{2\Gamma\left(\frac{d+1}{2}\right)}{\sqrt{\pi}\Gamma\left(\frac{d}{2}\right)}\,
    2^{-1/\alpha}
    \Gamma\left(1-\tfrac 1\alpha\right)
    \quad\text{and}\quad
    \Ee\left|B_{1}\right|
    = \frac{2\Gamma\left(\frac{d+1}{2}\right)}{\sqrt{\pi}\Gamma\left(\frac{d}{2}\right)}\, 2^{-1/2}\Gamma\left(\tfrac12\right).
\end{gather*}
Note that
\begin{gather*}
W_1(\mu_\alpha,\mu_2)
    = \inf_{\Pi\in\mathscr{C}(\mu_\alpha,\mu_2)}\iint |x-y|\,\Pi(\dif x,\dif y),
\end{gather*}
where $\mathscr{C}(\mu_\alpha,\mu_2)$ denotes the set of all couplings of $\mu_\alpha$
and $\mu_2$. Then we obtain
\begin{align*}
    W_1(\mu_\alpha,\mu_2)
    &\geq \inf_{\Pi\in\mathscr{C}(\mu_\alpha,\mu_2)} \left|\iint |x|\,\Pi(\dif x,\dif y) - \iint |y|\,\Pi(\dif x,\dif y)\right|\\
    &= \left|\int |x|\,\mu_\alpha(\dif x) - \int |y|\,\mu_2(\dif y) \right|\\
    &=\left| \Ee\left|\alpha^{-1/\alpha}L_{1}\right| - \Ee\left|2^{-1/2}B_{1}\right| \right|\\
    &=\frac{2\Gamma\left(\frac{d+1}{2}\right)}{\sqrt{\pi}\Gamma\left(\frac{d}{2}\right)}\,\left|(2\alpha)^{-1/\alpha}
    \Gamma\left(1-\tfrac1\alpha\right) - 2^{-1}\Gamma\left(\tfrac12\right)\right|.
\end{align*}
Combining this with Lemma \ref{gammalower} below, we complete the proof.
\end{proof}

\begin{lemma}\label{gammalower}
For any $\alpha\in (1,2)$,
\begin{gather*}
    \left|(2\alpha)^{-1/\alpha}\Gamma\left(1-\tfrac 1\alpha\right) - 2^{-1}\Gamma\left(\tfrac12\right)\right|
   \geq C(2-\alpha).
    \end{gather*}
\end{lemma}

\begin{proof}
Let
\begin{gather*}
    \phi(x)
    := (2x)^{-1/x}\Gamma\left(1-\frac 1x\right),
    \quad 1<x\leq2.
\end{gather*}
It is not hard to verify that for all $x\in(1,2)$, $\phi'(x)<\phi'(2-)<0$, and thus
\begin{align*}
    |\phi(\alpha)-\phi(2)|
    &= (2-\alpha)\left|\int_0^1\phi'\left(\alpha+r(2-\alpha)\right) \dif r \right|\\
    &=(2-\alpha) \int_0^1\left[-\phi'\left(\alpha+r(2-\alpha)\right) \right] \dif r\\
    &\geq -\phi'(2-)(2-\alpha).
    \qedhere
\end{align*}
\end{proof}

\begin{ack}
	C.-S.\ Deng is supported by     the       National Natural Science Foundation of China (12371149) and     the       Natural Science Foundation of Hubei Province of China (2022CFB129). 
	R.\ Schilling is supported through the DFG-NCN Beethoven Classic 3 project SCHI419/11-1 \& NCN 2018/31/G/ST1/02252,     the  6G-life project (BMBF programme ``Souver\"an. Digital. Vernetzt.'' 16KISK001K) and the SCADS.AI centre.      
	L.\ Xu is supported by National Natural Science Foundation of China No. 12071499, The Science and Technology Development Fund (FDCT) of Macau S.A.R. FDCT 0074/2023/RIA2, and University of Macau grants MYRG2020-00039-FST, MYRG-GRG2023-00088-FST.
\end{ack}

\end{document}